\documentclass{amsart}

\usepackage{amsmath}
\usepackage{amssymb}
\usepackage{enumerate}
\usepackage{amsbsy}
\usepackage{amsfonts}
\usepackage{color}

\newtheorem{theorem}{Theorem}[section]
\newtheorem{proposition}[theorem]{Proposition}

\newtheorem{lemma}[theorem]{Lemma}

\numberwithin{equation}{section}

\theoremstyle{remark}

\newcommand{\R}{\mathbb{R}}

\begin{document}

\title[Finite time blowup]{On finite time blowup for the mass-critical Hartree equations}
%\author[Short author names]{\textbf{Full author names in bold}\\Full postal addresses of all authors}

\author[Y. Cho, G. Hwang, S. Kwon and S. Lee]{\textbf{Yonggeun Cho}\\
Department of Mathematics, and Institute of Pure and Applied Mathematics, Chonbuk National University, Jeonju 561-756, Republic of Korea\\
(changocho@jbnu.ac.kr)\\
\textbf{Gyeongha Hwang}\\
Department of Mathematical Sciences, Ulsan National Institute of Science and Technology, Ulsan, 689-798, Republic of Korea\\
(ghhwang@unist.ac.kr)\\
\textbf{Soonsik Kwon}\\
Department of Mathematical Sciences, Korea Advanced Institute of Science and Technology , Daejeon 305-701, Republic of Korea\\
(soonsikk@kaist.edu)\\
\textbf{Sanghyuk Lee}\\
Department of Mathematical Sciences, Seoul National University, Seoul 151-747, Republic of Korea\\
(shklee@snu.ac.kr)\\}

%\MSdates[`Received date']{`Accepted date'}
\label{firstpage}
\maketitle

\begin{abstract}
We consider the fractional Schr\"odinger equations with focusing Hartree type nonlinearities. When the energy is negative, we show that the solution blows up in a finite time. For this purpose, based on Glassey's argument, we obtain a virial type inequality.
\end{abstract}


\maketitle

\vspace{1cm}

\section{Introduction}

\noindent In this paper we consider the Cauchy problem of the
focusing fractional nonlinear Schr\"odinger equations:
\begin{align}\label{main eqn}\left\{\begin{array}{l}
i\partial_tu = |\nabla|^\alpha u + F(u),\;\;
\mbox{in}\;\;\mathbb{R}^{1+n}
\times \mathbb{R},\\
u(x,0) = \varphi(x)\;\; \mbox{in}\;\;\mathbb{R}^n,
\end{array}\right.
\end{align}
where $|\nabla| = (-\Delta)^\frac12$, $n \ge 2$, $\alpha \ge 1$, and
$F(u)$ is a nonlocal nonlinear term of Hartree type given by
$$
F(u)(x) =
-\Big(\frac{\psi(\cdot)}{|\cdot|^\gamma}
* |u|^2\Big)(x)\,u(x)\\
\equiv - V_\gamma(|u|^2)(x) \,u(x)\,.
$$
       %where $*$ denotes the convolution in $\mathbb{R}^n$,
Here $0 \le \psi \in L^\infty(\mathbb R^n)$ and $0 < \gamma < n$. We
say that \eqref{main eqn} is focusing since $- V_\gamma(|u|^2)$
serves as an attractive self-reinforcing potential. We also use a
simplified notation $V_\gamma$ to denote $V_\gamma(|u|^2)$.

When $\psi$ is homogeneous of degree zero (e.g. $\psi\equiv 1 $),
the equation \eqref{main eqn} has scaling invariance. In fact, if
$u$ is a solution of \eqref{main eqn},  $u_\lambda$, $\lambda
> 0$, given by
$$
u_\lambda (t, x) = \lambda^{-\frac{\gamma-\alpha}2 + \frac n2} u(\lambda^\alpha \,t, \lambda x ),
$$
is also a solution. We denote the critical Sobolev exponent $s_c =
\frac{\gamma-\alpha}{2}$. Under scaling $u\to u_\lambda$,
$\dot{H}^{s_c}$-norm of data is preserved. The solution $u$ of
\eqref{main eqn} formally satisfies the mass and energy conservation
laws:
\begin{align}\begin{aligned}\label{consv}
m(u) &= \|u(t)\|^2_{L^2}, \\  %= \|\varphi\|^2_{L^2}
E(u) &=  K(u) + V(u), %= E(\varphi),
\end{aligned}\end{align}
where
$$
K(u) = \frac12 \big\langle u, |\nabla|^\alpha\,
u\big\rangle, \;\;V(u) = -\frac14 \big\langle u, V_\gamma(|u|^2)u
\big\rangle.
$$
Here $\big\langle\cdot,\cdot\big\rangle$ is the complex inner
product in $L^2$. In view of scaling invariance and the conservation
laws - when each conserved quantity is invariant under scaling - we
say the equation \eqref{main eqn} is mass-critical if $\gamma=
\alpha$ and energy-critical if $\gamma=2\alpha$.

The purpose of this paper is to show the finite time blow-up of
solutions to the fractional or higher order equations when
\eqref{main eqn} is mass-critical. If the energy is negative (i.e.
the magnitude of the potential energy $V(u)$ is larger than that of
kinetic part $K(u)$), then self-attracting power dominates  overall
dynamics and so it may result in a collapse in a finite time. For
the usual Schr\"odinger equations ($\alpha = 2$), Glassey \cite{gl}
introduced a convexity argument to show  existence of finite time
blow-up solutions. Indeed, if $\psi \equiv 1$, $2 \le \gamma <
\min(n, 4)$, $ n \ge 3$ and $\varphi \in
H^\frac\gamma2(\mathbb{R}^n)$ with $x\varphi \in L^2$, then
$$
\|\,x u(t)\|_{L^2}^2 \le 8t^2E(\varphi) + 4t\, \big\langle \varphi, A \varphi\,\big\rangle + \|x \varphi\|_{L^2}^2,
$$
where $A$ is the dilation operator $\frac1{2i}(\nabla \cdot x + x
\cdot \nabla)$. This implies that if $E(\varphi) <0 $, then the maximal
time of existence $T^* < \infty$. For details, see Section 6.5 in
\cite{caz} and Section \ref{prop of mom} below. %The focusing nonlinearity serves as an attracting potential.

In the fractional or high order equations, a variant of the second
moment is the quantity
$$\mathcal M(u) := \big\langle u, x \cdot |\nabla|^{2-\alpha} x u
\big\rangle.$$ This  was first utilized by Fr\"ohlich and Lenzmann
\cite{frohlenz2} in their study of  the semirelativistic nonlinear
Schr\"odinger equations ($\alpha=1$). More precisely, they obtained
$$
\mathcal M(u(t)) \le 2t^2E(\varphi) + 2t\big(\big\langle \varphi, A \varphi\big\rangle + C
\|\varphi\|_{L^2}^4\big)
+ \mathcal M(\varphi)
$$
for $\psi = e^{-\mu|x|}\;(\mu \ge 0)$, $\gamma = 1$, $\varphi \in
H_{rad}^2(\mathbb R^3)$ with $|x|^2 \varphi \in L^2$. Here the function space
$X_{rad}$ denotes the subspace $X$ of radial functions. The quartic term $\|\varphi\|_{L^2}^4$ appears due to the commutator $\big[|x|^2V_\gamma, |\nabla|\big]$,  and in $\mathbb R^3$ it is controlled by Newton's theorem.

Compared to the usual case ($\alpha=2 $), when $\alpha \ne 2$, the
presence of $ |\nabla|^{2-\alpha}$ gives rise to certain types of
singular integrals which necessitate use of commutators. So, the
main issue is how to estimate the commutator $\big[|x|^2V_\gamma,
|\nabla|^{2-\alpha}\big]$ since the Newton's theorem is generally
not available except for $\alpha=1$ in $\R^3$.
%We extend the commutator estimate to general cases $\big[|x|^2V_\gamma, |\nabla|^{2-\alpha}\big]$.
In order to obtain the desired estimate we use the Stein-Weiss
inequality \eqref{s-w ineq} and combine this with a convolution
estimate in Lemma \ref{weighted con}. To close our argument, we
further need an estimate for the moments $\|xu\|_{L^2}$ and
$\||x|\nabla u\|_{L^2}$ for $t$ contained in the existence time
interval. (See Proposition \ref{mom prop}.) It is done under some
regularity assumption\footnote{ Such an assumption is not necessary
for the usual Schr\"{o}dinger equation.} which we need to impose to
get estimates for the commutators $\big[\,|\nabla|^\alpha, |x|^2\big]$ and
$\big[\,V_\gamma, \nabla \cdot (x\cdot|\nabla|^{2-\alpha}x) \nabla \big]$.

The Hartree nonlinearity is essentially a cubic one, though it is
convolved with the potential. Thus, by fairly standard an  argument
one can show the local well-posedness of the Cauchy problem for
suitably regular initial data. Indeed, we have local well-posedness
for $s \ge \frac \gamma2$ so that, within the maximal existence time
interval $[0, T^*)$, there is a unique solution $u \in C([0,T^*);
H^s) \cap C^1([0,T^*); H^{s-\alpha})$ and $\lim\limits_{t\nearrow
T^*}\|u(t)\|_{H^\frac\gamma2} = \infty$ if $T^* < \infty$. For
convenience of readers, we append the local well-posedness for
general $\alpha > 0$ in the last section.

Let us define a Sobolev index $\alpha^*$ by $\alpha^* = (2k)^2$
where $k$ is the least integer satisfying $k \ge \frac\alpha2$. We
separately state our results for the low order case, $1\le\alpha <
2$, and the high order case, $2< \alpha < \frac n2+1$.
\begin{theorem}\label{blow}
Let $\gamma = \alpha$, $1 \le \alpha < 2$, and $n \ge 4$. Assume
that $\psi$ is a nonnegative, smooth, decreasing and radial function
with $|\psi'(\rho)| \le C\rho^{-1}$ for some $C>0$. Additionally,
assume that the initial datum satisfies $\varphi \in
H_{rad}^{\alpha^*}$ and $|x|^\ell \partial^\mathfrak j \varphi \in
L^2(\mathbb{R})$ for $1 \le \ell \le 2, 0 \le |\mathfrak j| \le
4-2\ell$. Then,  if $E(\varphi) < 0$, the solution to \eqref{main
eqn} blows up in a finite time.\end{theorem}

\begin{theorem}\label{blow2}
Let $\gamma = \alpha$, $2 < \alpha < 1 + \frac n2$ and $n \ge 4$.
Assume that $\psi$ is a nonnegative, smooth, decreasing and radial
function. Additionally, assume that the initial datum satisfies
$\varphi \in H_{rad}^{\alpha^*}$ and $|x|^{\ell} \partial^\mathfrak
j\varphi \in L^2(\mathbb{R})$ for $1 \le \ell \le 2k, 0 \le
|\mathfrak j| \le 2k(2k-\ell) $. Then,  if $E(\varphi) < 0$,  the
solution to \eqref{main eqn} blows up in a finite time.
\end{theorem}

The restriction $n \ge 4$ is due to the use of the Stein-Weiss
inequality. The technical condition $\alpha < 1 + \frac n2$ is
imposed because we make use of the convolution estimate
\eqref{weighted1} (Lemma \ref{weighted con}) and $ n\ge 4$. For the
proof of theorems we show that the mean dilation is decreasing when
$ E(\varphi) <0$. Clearly, this follows from
\begin{align}\label{mean-dilation}
\frac {d}{dt}\big\langle u, A u\big\rangle \le 2\alpha E(\varphi),
\end{align}
which holds whenever  $\gamma \ge \alpha$ and $\psi' \le 0$. If
$\gamma = \alpha$, from the estimate \eqref{2nd virial}, we have,
for $t \in [0, T^*)$,
\begin{align}\begin{aligned}\label{mom}
\mathcal M(u) \le 2\alpha^2 t^2E(\varphi) + 2\alpha t\big(\big\langle \varphi, A \varphi\big\rangle + C
\|\varphi\|_{L^2}^4\big)
+ \mathcal M(\varphi).
\end{aligned}\end{align}
In order to validate \eqref{mean-dilation} and \eqref{mom}, we need
estimates for the moments $\|x u\|_{L^2}$ and $\||x|\nabla
u\|_{L^2}$ on the time interval $[0,T^*)$.

We finally remark that the argument of this paper dose not readily
work for the power type nonlinearity. Since our argument relies on
$H^{\alpha^*}$ regularity assumption and the Stein-Weiss inequality, a
different approach seems to be necessary in order to control the
commutators.

The rest of paper is organized as follows. In Section 2 we show
the finite time blow-up while assuming Proposition \ref{mom prop}. In
Section \ref{prop of mom} we provide the proof of Proposition
\ref{mom prop}. The last section is devoted to the local well-posedness.

\subsection*{Notations}
\noindent We use the notations:
$\partial^{\mathfrak j} = \prod_{1 \le i \le n}\partial_i^{j_i}$ for
multi-index $\mathfrak j = (j_1, \cdots, j_n)$ and $|\mathfrak
j|=\sum_i j_i$. $|\nabla| = \sqrt{-\Delta}$, $\dot
H_r^s = |\nabla|^{-s}L^r$, $\dot H^s=\dot H_2^s$, and $H_r^s = (1 -
\Delta)^{-s/2} L^r$, $H^s = H_2^s$. $A \lesssim B$ and $A \gtrsim B$
means that $A \le CB$ and $A \ge C^{-1}B$, respectively, for some
$C>0$. As usual $C$ denotes a positive constant, possibly depending
on $n, \alpha$ and $\gamma$, which may differ at each occurrence.

\section{Finite time blow-up}
In this section we consider finite time blow-up of solutions to the
Cauchy problem \eqref{main eqn} of the mass-critical potentials. We
begin with the dilation operator $A$. With more general assumption
for $\psi$ and $\gamma$ we obtain an estimate for the time evolution
of average of $A$.
\begin{lemma}\label{virial-lemma1}
 Let $\psi$ be radially
symmetric smooth function such that $\psi' = \partial_r\psi\le 0$.
Suppose that $u \in H^{\alpha}$ and $x u(t), |x|  \nabla u(t) \in
L^2$ for $t \in [0, T^*)$, where $T^*$ is the maximal existence
time. Then, for $\gamma \ge \alpha$,
\begin{align}\label{1st virial}
\frac{d}{dt}\big\langle u, A u\big\rangle \le 2\alpha E(\varphi).
\end{align}
\end{lemma}
\begin{proof}
Since $u \in H^{\alpha}$ and $|x|u, x\cdot \nabla u \in L^2$,
$\big\langle u, A u\big\rangle$ is well-defined and so is
\begin{align}\label{commut0}
\frac{d}{dt}\big\langle u, A u\big\rangle = i\big\langle u, \big[H, A\big] u\big\rangle,
\end{align}
where $H = |\nabla|^\alpha - V_\gamma$. Here $\big[H, A\big]$ denotes the
commutator $HA - AH$. Using the identity $\big[|\nabla|^\alpha, x\big] =
 -\alpha |\nabla|^{\alpha-2} \nabla$, we have
\begin{align}\label{commut1}
\big[\,|\nabla|^\alpha, A\big] = -i \alpha |\nabla|^\alpha.
\end{align}
Similarly,
\begin{align}\label{commut2}
\big[\, - V_\gamma, A\big] = -i (x \cdot \nabla) V_\gamma.
\end{align}
Substituting \eqref{commut1} and \eqref{commut2} into \eqref{commut0}, we get
\begin{align}\label{virial0}
\frac{d}{dt}\big\langle u, A u\big\rangle = \alpha\big\langle u, |\nabla|^\alpha
u \big\rangle + \big\langle u, (x \cdot \nabla) V_\gamma u\big\rangle.
\end{align}
For Hartree type $V_\gamma$, we have
\begin{align*}
(x \cdot \nabla) V_\gamma &= -\gamma \int \frac{\psi(|x-y|)}{|x-y|^{\gamma}}|u(y)|^2\,dy + \int \frac{\psi'(|x-y|)}{|x-y|^{\gamma}}|x-y||u(y)|^2\,dy\\
&\qquad- \int \big(\gamma \frac{\psi(|x-y|)}{|x-y|^{\gamma+1}} -
\frac{\psi'(|x-y|)}{|x-y|^{\gamma}}\big)\frac{y\cdot
(x-y)}{|x-y|}|u(y)|^2\,dy,
\end{align*}
\begin{align*}
\big\langle u, (x \cdot \nabla)V_\gamma u\big\rangle &=  4\gamma V(u) + \int\!\!\!\int \frac{|x-y|\psi'(|x-y|)}{|x-y|^{\gamma}}|u(x)|^2|u(y)|^2\,dxdy\\
&- \big\langle u, (x \cdot \nabla) V_\gamma u\big\rangle,
\end{align*}
which implies
\begin{align*}
\big\langle u, (x \cdot \nabla) V_\gamma u\big\rangle &= 2\gamma V(u) +
\frac12\int\!\!\!\int \frac{|x-y|
\psi'(|x-y|)}{|x-y|^{\gamma}}|u(x)|^2|u(y)|^2\,dxdy.
\end{align*}
Substituting this into \eqref{virial0} gives
\begin{align*}
\frac{d}{dt}\big\langle u, A u\big\rangle &\le 2\alpha E(\varphi) +
2(\gamma-\alpha)V(u)\\
&\qquad\qquad + \frac12\int\!\!\!\int\big(
|x-y|\psi'(|x-y|)\big)\frac{|u(x)|^2|u(y)|^2}{|x-y|^\alpha}\,dxdy.
\end{align*}
Since $\gamma \ge \alpha$ and $\psi'(|x|) \le 0$, we get \eqref{1st
virial}. This completes the proof of Lemma \ref{virial-lemma1}.
\end{proof}
Next we consider the nonnegative quantity $\mathcal M(u) = \big\langle u, M u \big\rangle$ with the virial operator
$$
M := x \cdot |\nabla|^{2-\alpha} x = \sum_{k = 1}^n x_k |\nabla|^{2-\alpha}x_k.
$$
Suppose $u(t) \in H^{\alpha^*}$ and $|x|^{2k}u(t) \in L^2$ for $t
\in [0, T^*)$. Then, since $\mathcal M(u) \lesssim
\||x|\nabla u\|_{L^2}\|(1 + |x|)^{2k}u\|_{L^2}$, from \eqref{appr3} below the quantity
$\mathcal M(u)$ is well-defined and finite for all $t \in [0,T^*)$,
and so is
\begin{align}\label{virial1}
\frac{d}{dt}\mathcal M(u) = i\big\langle u, \big[H, M\big]u\big\rangle  =
i\big\langle u, \big[\,|\nabla|^\alpha, M\big] u\big\rangle -i\big\langle u, \big[V_\gamma,
M\big]u\big\rangle.
\end{align}

\begin{lemma}\label{virial-lemma3}
Suppose that $u(t) \in H^{\alpha^*}$ and $|x|^{2k}u(t) \in L^2$ for $t \in [0, T^*)$.
Then we have
\begin{align}\label{2nd virial}
\frac{d}{dt}\mathcal M(u) \le 2\alpha\big\langle u, A u \big\rangle +
C\|\varphi\|_{L^2}^4
\end{align}
for $t \in [0, T^*)$, where $C$ is a positive constant depending only on $n, \alpha$ but not on $u, \varphi$.
\end{lemma}
Now, Theorem \ref{blow} and Theorem \ref{blow2} immediately  follow
from Lemma \ref{virial-lemma1} and \ref{virial-lemma3} once we have
Proposition \ref{mom prop}.

\begin{proof} By the identity $|\nabla|^\alpha x = x|\nabla|^\alpha  - \alpha
|\nabla|^{\alpha-2}  \nabla$, we have
\begin{align*}
\big[\,|\nabla|^\alpha, M\big] = |\nabla|^\alpha x \cdot|\nabla|^{2-\alpha} x - x \cdot|\nabla|^{2-\alpha}x |\nabla|^\alpha = -\alpha(x\cdot \nabla + \nabla \cdot x).
\end{align*}
Hence, for a smooth function $v$ we get
\begin{align*}
\big[v, M\big] &= v x \cdot|\nabla|^{2-\alpha} x - x\cdot|\nabla|^{2-\alpha}x v\\
 &= v|x|^2|\nabla|^{2-\alpha} - (2-\alpha)vx \cdot \nabla |\nabla|^{-\alpha} - |\nabla|^{2-\alpha}|x|^2v - (2-\alpha)|\nabla|^{-\alpha}\nabla \cdot x v\\
 &= \big[\,|x|^2v, |\nabla|^{2-\alpha}\big]  + (\alpha-2)\big(vx\cdot \frac{\nabla}{|\nabla|}|\nabla||\nabla|^{-\alpha} + |\nabla||\nabla|^{-\alpha}\frac{\nabla}{|\nabla|} \cdot x v\big).
\end{align*}
By density argument we may assume that $v = V_\alpha$ in the above
identity. Thus it suffices to show that
\begin{align}
\label{critical-virial}
\begin{aligned}
\|\varphi\|_{L^2}^4&\gtrsim  \big|\big\langle u, \big[|x|^2V_\alpha,
|\nabla|^{2-\alpha}\big] u \big\rangle\big|\,\, +\\& \big|\big\langle u,
\big(V_\alpha x\cdot \frac{\nabla}{|\nabla|}|\nabla|
|\nabla|^{-\alpha} +
|\nabla||\nabla|^{-\alpha}\frac{\nabla}{|\nabla|} \cdot x
V_\alpha\big)u\big\rangle\big|.
\end{aligned}
\end{align}

\noindent{\it Case $\alpha > 2$.} We first consider the higher order
case $\alpha > 2$.   The first term of LHS of
\eqref{critical-virial} is rewritten as
\begin{align}\label{i}
2\big|\mbox{Im}\big\langle u, |x|^2V_\alpha |\nabla|^{2-\alpha}u\big\rangle\big|.
\end{align}
To handle this we recall the following
weighted convolution estimate (see \cite{chozsash, chonak}):
\begin{lemma}\label{weighted con}
Let $0 < \gamma < n-1$ and $n \ge 2$.  Then, for any $f\in
L_{rad}^1$ and $x \neq 0$,
\begin{align}\label{weighted1}
\int |x-y|^{-\gamma}|f(y)|\,dy \lesssim |x|^{-\gamma}\|f\|_{L^1}.
 \end{align}
\end{lemma}
From Lemma \ref{weighted con} and mass conservation \eqref{i} is bounded by
\begin{align}\label{1st term}
C\|\psi\|_{L^\infty}\|\varphi\|_{L^2}^2\int |u(x)||x|^{-(\alpha - 2)}\int |x-y|^{-(n-\alpha+2 )}|u(y)|\,dydx.
\end{align}
To estimate this, we make use of the following inequality due to
Stein-Weiss \cite{s-w}:
\\
For $f \in L^p$ with $1 < p < \infty$, $0 <
\lambda < n$, $\beta < \frac np$, and $n = \lambda + \beta$
\begin{align}\label{s-w ineq}
\||x|^{-\beta} (|\cdot|^{-\lambda}*f)\|_{L^p} \lesssim \|f\|_{L^p}.
\end{align}
Applying \eqref{s-w ineq} with $p = 2$, $\beta = \alpha - 2$ and $\lambda = n-(\alpha-2)$, \eqref{1st term} is bounded by $C\|\varphi\|_{L^2}^4$.

We write the second term of RHS of \eqref{critical-virial}  as
$$
2\big|\mbox{Im}\big\langle u, V_\alpha x\cdot
\frac{\nabla}{|\nabla|}|\nabla| |\nabla|^{-\alpha} u\big\rangle\big|.
$$
By using Lemma \ref{weighted con} we see that this is bounded by
$$
C\|\psi\|_{L^\infty}\|\varphi\|_{L^2}^2\int
|u(x)||x|^{-(\alpha-1)}\int |x-y|^{-(n-(\alpha-1))}
\big|\big(\frac{\nabla}{|\nabla|}u\big)(y)\big|\,dydx.
$$
Applying \eqref{s-w ineq} with $p =2$, $\beta = \alpha-1$ and
$\lambda = n-(\alpha-1)$, and Plancherel's theorem, we get the
desired bound \eqref{critical-virial}.

\smallskip

 \noindent{\it Case $1 \le \alpha < 2$.} Now we consider the
fractional case $1 \le \alpha < 2$. The second term of RHS of
\eqref{critical-virial} can be treated in the same way as the high
order case, and it is bounded by $C\|\varphi\|_{L^2}^4$. Hence, it
suffices to consider the first term. Let us set $g = |x|^2V_\alpha$.
Then, we need only to obtain
\begin{align}\label{L2-bound}
\|\big[|\nabla|^{2-\alpha}, g\big]u\|_{L^2} \le C \|\varphi\|_{L^2}^3,
\end{align}
which gives $\big|\big\langle u, \big[|x|^2V_\alpha,
|\nabla|^{2-\alpha}\big] u \big\rangle\big|\lesssim
\|\varphi\|_{L^2}^4$, and thus \eqref{critical-virial}.  The kernel
$K(x, y)$ of the commutator $\big[|\nabla|^{2-\alpha}, g\big]$ can be
written as $k(x-y)(g(y)-g(x))$, where $k$ is the kernel of
pseudo-differential operator $|\nabla |^{2-\alpha}$. Let $K^\ast$ be
the kernel of the dual operator of $\big[|\nabla|^{2-\alpha}, g\big]$. Then,
obviously $K^*(x, y) = -K(x, y)$.

Suppose that $\|g\|_{\dot \Lambda^{2-\alpha}} = \sup_{x \neq y \in
\mathbb R^n}\frac{|g(x) - g(y)|}{|x-y|^{2-\alpha}} < \infty$. Since
$|k(x-y)| \lesssim |x-y|^{-n-(2-\alpha)}$, $|\nabla k(x-y)|\lesssim
|x-y|^{-n-1-(2-\alpha)}$, and $0<2-\alpha\le 1$,  it is easy to see
the following:
\begin{align*}
|K(x, y)| &\lesssim |x-y|^{-n},\\
|K(x, y) - K(x',y)| &\lesssim \frac{|x-x'|^{2-\alpha}}{|x-y|^{n+2-\alpha}},\;\;\mbox{if}\;\;|x-x'| \le |x-y|/2,\\
|K(x, y) - K(x,y')| &\lesssim \frac{|y-y'|^{2-\alpha}}{|x-y|^{n+2-\alpha}},\;\;\mbox{if}\;\;|y-y'| \le |x-y|/2,
\end{align*}
and so does $K^*$ (because $K^*(x, y) = -K(x, y)$). Let $\zeta$ be a
normalized bump function supported in the unit ball and set
$\zeta^{x_0, N}(x) = \zeta((x-x_0)/N)$.
 By Theorem
3 in p. 294 of \cite{st}, in order to prove \eqref{L2-bound}, it is
sufficient to show that
\begin{align}\label{rest-bounded}
\|\big[|\nabla|^{2-\alpha}, g\big](\zeta^{x_0, N})\|_{L^2} \le C \|\varphi\|_{L^2}^2 N^\frac n2
\end{align}
with $C$, independent of $x_0, N, \zeta$.

We now show \eqref{rest-bounded}. The commutator
$\big[|\nabla|^{2-\alpha}, g\big]$ can be written as
\begin{align}\label{sum}
\sum_{j = 1}^n \big[T_j, g\big]\partial_j + \sum_{j = 1}^n T_j (\partial_j g),
\end{align}
where $T_j = - |\nabla|^{2-\alpha}(-\Delta)^{-1}\partial_j$. For the first sum of \eqref{sum} we obtain
\begin{align}\label{cz-est}
\|\big[T_j, g\big]\partial_j (\zeta^{x_0, N})\|_{L^2} \le
C\|\varphi\|_{L^2}^2N^\frac n2.
\end{align}
Indeed, let $k_j$ be the kernel of $T_j$. If $\alpha = 1$, $k_j$ is
the kernel of Riesz transform. If $1 < \alpha < 2$, it is easy to
see $|k_j(x, y)| \lesssim |x-y|^{-n + \alpha-1}$ (note that
$|\widehat{k_j}(\xi)| \lesssim |\xi|^{-(\alpha-1)}$). Thus it
follows that
$$
|K_j(x, y)| = |k_j(x-y)||g(y) - g(x)| \lesssim \|g\|_{\dot\Lambda^{2-\alpha}}|x-y|^{-(n-1)}.
$$
Hence, for $|x-x_0| < 2N$, $ |\big[T_i, g\big]\partial_i (\zeta^{x_0,N})(x)|
\lesssim \|g\|_{\dot\Lambda^{2-\alpha}}. $ This gives \[\|\big[T_i,
g\big]\partial_i (\zeta^{x_0, N})\|_{L^2(\{|x-x_0| < 2N\})} \lesssim
\|g\|_{\dot\Lambda^{2-\alpha}}N^\frac n2.\] If $|x-x_0| \ge 2N$, we
have $ |\big[T_i, g\big]\partial_i (\zeta^{x_0, N})(x)| \lesssim
\|g\|_{\dot\Lambda^{2-\alpha}}N^{n-1}|x-x_0|^{-(n-1)}. $ Hence,
\begin{align*}
\|\big[T_i, g\big]\partial_i (\zeta^{x_0, N})\|_{L^2(\{|x-x_0| \ge 2N\})} &\lesssim
 \|g\|_{\dot\Lambda^{2-\alpha}}N^{n-1}\big(\int_{|x| > 2N} |x|^{-2(n-1)}\,dx\big)^\frac12 \\
&\lesssim \|g\|_{\dot\Lambda^{2-\alpha}}N^\frac n2.
\end{align*}
We now show $\|g\|_{\dot\Lambda^{2-\alpha}} \le C
\|\varphi\|_{L^2}^2 $, which gives \eqref{cz-est}. If $x \neq y$,
then
$$
|g(x) - g(y)| \le |x-y|\int_0^1 |\nabla g (z_s)|\,ds, \quad z_s = x + s(y-x).
$$
Since $|\psi'(\rho)| \le C\rho^{-1}$ for $\rho > 0$, from Lemma \ref{weighted con} and mass conservation it follows that
$$
 |\nabla g (z_s)| \lesssim |z_s|^{1-\alpha}\|u\|_{L^2}^2 = ||x|-s|x-y||^{1-\alpha}\|\varphi\|_{L^2}^2,
$$
provided $\alpha < n-2$. Since $\sup_{a > 0}\int_0^1|a -
s|^{-\theta}\,ds \le C_\theta$ for $0 < \theta < 1$, we have $ |g(x)
- g(y)| \lesssim |x-y|^{2-\alpha}\|\varphi\|_{L^2}^2 $. Thus we get \eqref{cz-est}.

Finally we need to handle the second sum of \eqref{sum}. If $\alpha
= 1$, $T_j$ is a Riesz transform. Thus,
\begin{align*}
\| T_j ((\partial_jg) \zeta^{x_0, N}) \|_{L^2} \le C\|\partial_j g\|_{L^\infty}N^\frac n2.
\end{align*}
By Lemma \ref{weighted con} for $\alpha =1$, we get $|\partial_j
g(x)| \le |x|V_1 + |x|^2|\partial_jV_1| \lesssim
\|\varphi\|_{L^2}^2$. Hence,
 \begin{align}\label{rest-bounded2}
\| T_j ((\partial_jg) \zeta^{x_0, N}) \|_{L^2} \le C\|\varphi\|_{L^2}^2N^\frac n2.
\end{align}
For $1 < \alpha < 2$, the kernel $k_j(x)$ of $T_j$ is bounded by
$C|x|^{-(n-\alpha+1)}$. So, from the duality and Lemma \ref{weighted
con} with $\alpha < n-2$ we have, for any $\psi \in L^2$,
\begin{align*}
\big|\big\langle \psi, T_j ((\partial_jg) \zeta^{x_0, N}) \big\rangle\big| &= \big|\big\langle T_j^* \psi, (\partial_j g) \zeta^{x_0, N}\big\rangle\big|\\
&\le CN^\frac n2\||\partial_j g(\cdot)|\int |\cdot-y|^{-(n-\alpha+1)}|\psi(y)|\,dy\|_{L^2}\\
&\le CN^\frac n2\|\varphi\|_{L^2}^2\||\cdot|^{1-\alpha}|\int |\cdot-y|^{-(n-\alpha+1)}|\psi(y)|\,dy\|_{L^2},
\end{align*}
where $T_j^*$ is the dual operator of $T_j$. Using  \eqref{s-w ineq}
with $\beta = \alpha-1$, $\lambda = n-\alpha+1$ and $ p =2$,
we get
\begin{align*}
\big|\big\langle \psi, T_j (\partial_j g \zeta^{x_0, N})
\big\rangle\big| \le C\|\psi\|_{L^2}\|\varphi\|_{L^2}^2N^\frac n2.
\end{align*}
Thus, it follows that
\begin{align}\label{comm-error}
\|T_j (\partial_j g \zeta^{x_0, N})\|_{L^2} \le C\|\varphi\|_{L^2}^2N^\frac n2.
\end{align}
Therefore combining the estimates \eqref{cz-est}
\eqref{rest-bounded2} and \eqref{comm-error} yields
\eqref{rest-bounded}. This completes the proof of Lemma
\ref{virial-lemma3}.
\end{proof}

\section{Propagation of the moment}\label{prop of mom}

We now discuss estimates for the propagation of  moments
$\||x|^{2k}u\|_{L^2}$ when $|x|^{2k}\varphi \in L^2$ and the
solution $u \in C([0,T^*); H^{\alpha^*})$. For $\alpha < 2k$, we use
the kernel estimate of Bessel potentials. Let us denote,
respectively, the kernels of Bessel potential $D^{-\beta}$ and
$|\nabla|^\alpha D^{-2k}$ $(\beta = \alpha - 2k)$ by $G_\beta(x)$
and $K(x)$, where $D = \sqrt{1 - \Delta}$. Then
$$
K(x) = \sum_{j=0}^\infty A_j G_{2j+\beta}(x),
$$
where the coefficients $A_j$ are given by the
expansion $(1-t)^\frac\alpha2 = \sum_{j = 0}^\infty A_j t^j$ for
$|t| < 1$ with $\sum_{j \ge 0}|A_j| < \infty$. One can show that
$(1+ |x|)^\ell K \in L^1$ for $\ell \ge 1$ and has
decreasing radial and integrable majorant. In fact, from the
integral representation
$$
G_{2j+\beta}(x) = \frac1{(4\pi)^{n/2}\Gamma(j+\beta/2)}\int_0^\infty
\lambda^{(2j+\beta-n)/2-1} e^{-|x|^2/4\lambda} e^{ -
\lambda}\,d\lambda,
$$
it follows that,  for $j$ with $2j+\beta < n$,
\begin{align}\label{bessel1}
G_{2j+\beta}(x) \le C( |x|^{-n + 2j + \beta}\chi_{\{|x| \le 1\}}(x) +
e^{-c|x|}\chi_{\{|x| > 1\}}(x)),
\end{align}
and, for $j$ with $2j+\beta \ge n$,
\begin{align}\label{bessel2}
G_{2j+\beta}(x) \le C(\chi_{\{|x| \le 1\}}(x) + e^{-c|x|}\chi_{\{|x|
> 1\}}(x)).
\end{align}
Here the constants $c$ and $C$ of \eqref{bessel1} and
\eqref{bessel2} are independent of $j$. So, the function
$(1+|x|)^\ell G_{2j + \beta}$ has a decreasing radial and integrable
majorant, which is chosen uniformly on $j$, and so does $K$. For
details see p.132--135 of \cite{st-sing}.

\begin{proposition}\label{mom prop}
Let\, $T^*$ be the maximal time of solution $u \in
C([0,T^*);H^{\alpha^*})$ to \eqref{main eqn}. If\, $|x|^{\ell}\partial^\mathfrak j\varphi
\in L^2(\mathbb{R})$ for $1 \le \ell \le 2k, 0 \le |\mathfrak j| \le 2k(2k-\ell)$, then $|x|^{\ell}\partial^\mathfrak j u(t) \in L^2(\mathbb{R})$ for all $t \in [0, T^*)$.
\end{proposition}

\newcommand{\mm}{\mathbf{m}}
\newcommand{\wmm}{\widetilde{\mathbf{m}}}
\newcommand{\pe}{\psi_{\varepsilon}}

Let us set $\pe(x) = e^{-\varepsilon|x|^2}$. For proof of Proposition \ref{mom prop} we use the
following bootstrapping lemma.

\begin{lemma}\label{boot}
Let $\ell, m$ be integers such that  $2 \le \ell \le 2k$ and $0 \le
m \le \alpha^* - 2k$. Suppose that $\sup_{0 \le t' \le
t}(\|u(t')\|_{H^{2k+m}} + \||x|^{j}\partial^\mathfrak j
u(t')\|_{L^2}) < \infty$ for all $t \in [0, T^*)$ and $0 \le j \le
\ell-1, |\mathfrak j| \le 2k+m$. Then $\sup_{0\le t' \le
t}\||x|^\ell \partial^\mathfrak m u(t')\|_{L^2} < \infty$  for all
$t \in [0, T^*)$ and $|\mathfrak m| = m$.
\end{lemma}

\begin{proof}
Let $v = \partial^\mathfrak m u$ and
$$
\mm_\varepsilon(t) = \big\langle v(t), |x|^{2\ell} \pe^2
v(t) \big\rangle.
$$
From the regularity of the solution $u$ it follows that
\begin{align*}
\mm_\varepsilon'(t) = 2{\rm Im} \big\langle v, \big[|\nabla |^\alpha, |x|^{2\ell} \pe^2\big] v\big\rangle + 2{\rm Im} \big\langle |x|^\ell \pe v, |x|^\ell \pe \partial^\mathfrak m (V_\alpha u)\big\rangle =: 2(\,I + I\!I\,).
\end{align*}
We first prove the case, $\alpha < 2k$. We rewrite $I$ as
\begin{align*}
I &= {\rm Im }\big\langle |x|^\ell \pe v, \big[|\nabla|^\alpha D^{-2k}, |x|^\ell \pe\big] D^{2k}u\big\rangle\\
 &\qquad + {\rm Im}\big\langle |\nabla|^\alpha D^{-2k}(|x|^\ell \pe v), \big[D^{2k}, |x|^\ell \pe\big]v\,\big\rangle =: I_1 + I_2.
\end{align*}
By the kernel representation of $|\nabla|^\alpha D^{-2k}$, we  have
\begin{align*}
\big|\big[|&\nabla D^{-2k}, |x|^\ell \pe|\big]D^{2k}u (x)\big|\\
 &\le \int K(x-y)||x|^\ell \pe(x)- |y|^\ell \pe(y)||D^{2k}u(y)|\,dy\\
&\lesssim \int K(x-y)|x-y|(|x|^{\ell-1} + |y|^{\ell-1}) |D^{2k}u(y)|\,dy\\
&\lesssim \int K(x-y)|x-y|^{\ell}|D^{2k}u(y)|\,dy +  \int K(x-y)|x-y||y|^{\ell-1}|D^{2k}u(y)|\,dy.
\end{align*}
Since $|x|^\ell K$ is integrable, Cauchy-Schwarz inequality gives
$$
I_1 \lesssim \sqrt{\mm_\varepsilon}\,(\|u\|_{H^{2k}} +
\||x|^{\ell-1}D^{2k}u\|_{L^2}).
$$
As for $I_2$
we have
\begin{align*}
I_2 &= \sum_{1 \le j \le k}c_j {\rm Im}\big\langle |\nabla|^\alpha D^{-2k}(|x|^\ell \pe v),
 \big[\Delta^j, |x|^\ell \pe \big]v\,\big\rangle\\
&= \sum_{1 \le j \le k}c_j{\rm Im}\big\langle |\nabla|^\alpha D^{-2k}(|x|^\ell \pe v),
\sum_{\substack{ |\mathfrak j_1|+ |\mathfrak j_2| + |\mathfrak j_3| = 2j\\ 0
\le |\mathfrak j_3| \le 2j-1}}c_{\mathfrak j_1, \mathfrak j_2, \mathfrak j_3}\partial^{\mathfrak j_1}(|x|^\ell)
\partial^{\mathfrak j_2} \pe \partial^{\mathfrak j_3}
v\,\big\rangle.
\end{align*}
Note  that $|\partial^{\mathfrak j_1}(|x|^\ell)| \lesssim |x|^{\ell
- |\mathfrak j_1|}$ and $|\partial^{\mathfrak j_2} \pe (x)| \lesssim
\varepsilon^\frac{|\mathfrak j_2|}2(1 +
\varepsilon|x|^2)^\frac{|\mathfrak j_2|}2\pe(x)$. Hence, it follows
that
\begin{align*}
I_2 &\lesssim \||\nabla|^\alpha D^{-2k}(|x|^\ell \pe v)\|_{L^2}\sum_{1 \le j \le k}\big(\sum_{\substack{ |\mathfrak j_1|+ |\mathfrak j_2| + |\mathfrak j_3| = 2j\\ 0 \le |\mathfrak j_3| \le j-\ell}} + \sum_{\substack{ |\mathfrak j_1|+ |\mathfrak j_2| + |\mathfrak j_3| = j\\ j-\ell \le |\mathfrak j_3| \le 2j-1}}\big)\||x|^{\ell - |\mathfrak j_1|- |\mathfrak j_2|}\partial^{\mathfrak j_3}v\|_{L^2}\\
&\lesssim \sqrt{\mm_\varepsilon}\sum_{1 \le j \le k}\big(\sum_{\substack{ |\mathfrak j_1|+ |\mathfrak j_2| + |\mathfrak j_3| = 2j\\ 0 \le |\mathfrak j_3| \le j-\ell}} + \sum_{\substack{ |\mathfrak j_1|+ |\mathfrak j_2| + |\mathfrak j_3| = 2j\\ j-\ell \le |\mathfrak j_3| \le 2j-1}}\big)\||x|^{|\mathfrak j_3|-(j-\ell)}\partial^{\mathfrak j_3}v\|_{L^2}.
\end{align*}
Here we use the fact that the kernel of $|\nabla|^\alpha D^{-2k}$ is
integrable. Conventionally, the summand  is zero if
$j-\ell < 0$. By the Hardy-Sobolev inequality we get, for $0 \le
|\mathfrak j_3| \le j-\ell$,
$$\||x|^{|\mathfrak j_3|-(j-\ell)}\partial^{\mathfrak j_3}v\|_{L^2} \lesssim \|\partial^{\mathfrak j_3}v\|_{H^{j-\ell-|\mathfrak j_3|}} \lesssim \|v\|_{H^{j-\ell}} \lesssim \|u\|_{H^{j-\ell+m}}.$$
If $j-\ell \le |\mathfrak j_3| \le 2j-1$, then
$$
\||x|^{|\mathfrak j_3|-(j-\ell)}\partial^{\mathfrak j_3}v\|_{L^2} = \||x|^{|\mathfrak j_3|-(j-\ell)} \partial^{\mathfrak j_3 + \mathfrak m}u\|_{L^2}.
$$
Thus we finally obtain
\begin{align}\label{appr1}
I \lesssim \sqrt{\mm_\varepsilon(t)}\Big(\|u(t)\|_{H^{2k+m}} + \sum_{0 \le |\mathfrak j| \le 2k+m}\|(1 + |x|)^{\ell-1}\partial^\mathfrak j u(t)\|_{L^2}\Big).
\end{align}
In the case $\alpha = 2k$, we do not need the estimate for $I_1$. For the estimate of $I_2 = {\rm Im}\big\langle |x|^\ell \pe v, \big[\Delta^k, |x|^\ell
\pe\big]v \big\rangle$, we estimate similarly to obtain \eqref{appr1}.

Now we proceed to estimate $I\!I$. For this let us observe that
\begin{align*}
I\!I &= \sum_{\substack{\mathfrak m_1 + \mathfrak m_2 = \mathfrak m \\ 0 \le |\mathfrak m_2| \le m-1}} c_{\mathfrak m_1, \mathfrak m_2} {\rm Im}\big\langle |x|^\ell \pe v, |x|^\ell \pe \partial^{\mathfrak m_1} V_\alpha \partial^{\mathfrak m_2} u\big\rangle\\
&\lesssim \sqrt{\mm_\varepsilon}\,\sum_{\substack{\mathfrak m_1 + \mathfrak m_2 = \mathfrak m \\ 0 \le |\mathfrak m_2| \le m-1}}
\||x|\partial^{\mathfrak m_1}V_\alpha\|_{L^\infty}\||x|^{\ell-1}|\partial^{\mathfrak m_2} u\|_{L^2}.
\end{align*}
By Young's inequality we estimate
\begin{align*}
|x||\partial^{\mathfrak m_1}V_\alpha| &\lesssim \sum_{\mathfrak m_1^1 + \mathfrak m_1^2= \mathfrak m_1}\int |x-y|^{-\alpha}(|x-y| + |y|)|\partial^{\mathfrak m_1^1}u(y)||\partial^{\mathfrak m_1^2}u(y)|\,dy\\
&\lesssim \sum_{\mathfrak m_1^1 + \mathfrak m_1^2= \mathfrak m_1}\left( \int |x-y|^{-(\alpha-1)}(|\partial^{\mathfrak m_1^1}u(y)|^2 + |\partial^{\mathfrak m_1^2}u(y)|^2)\,dy \right.\\
&\qquad\qquad\qquad \left. + \int |x-y|^{-\alpha}(|y|^2|\partial^{\mathfrak m_1^1}u(y)|^2 + |\partial^{\mathfrak m_1^2}u(y)|^2)\,dy \right).
\end{align*}
Using the Hardy-Sobolev inequality, we get
\begin{align}\label{appr2}
I\!I \lesssim \sqrt{\mm_\varepsilon}\,\sum_{0 \le |\mathfrak j| \le k + m}\|(1+|x|)^{\ell-1}\partial^\mathfrak j u\|_{L^2}^3.
\end{align}
From \eqref{appr1} and \eqref{appr2} it follows that
$$
\mm_\varepsilon'(t) \le \sqrt{\mm_\varepsilon(t)}\Big(\|u(t)\|_{H^{2k+m}} + \sum_{0 \le |\mathfrak j| \le 2k+m}(1+ \|(1 + |x|)^{\ell-1}\partial^\mathfrak j u(t)\|_{L^2})^3\Big),
$$
which implies
$$
\sqrt{\mm_\varepsilon(t)} \lesssim \sqrt{\mm_\varepsilon(0)} + \int_0^t\Big(\|u(t')\|_{H^{2k+m}} + \sum_{0 \le |\mathfrak j| \le 2k+m}(1+ \|(1 + |x|)^{\ell-1}\partial^\mathfrak j u(t')\|_{L^2})^3\Big)\,dt'.
$$
Letting $\varepsilon \to 0$, by Fatou's lemma we get $\sup_{0 \le t' \le t} \||x|^\ell \partial^\mathfrak m u\|_{L^2} < \infty$ for all $t \in [0,T^*)$.
\end{proof}

\begin{proof}[Proof of Proposition \ref{mom prop}]
In view of  Lemma \ref{boot} it suffices to show that
\begin{align}\label{appr3}
\sup_{0\le t' \le t}\||x|\partial^{\mathfrak j}u(t')\|_{L^2} <
\infty\;\;\mbox{for all}\;\;|\mathfrak j| \le
\alpha^*-2k\;\;\mbox{and}\;\; t \in [0, T^*),
\end{align}
provided $u \in C([0,T^*); H^{\alpha^*})$. In fact, we can
obtain the same estimates as \eqref{appr1} and \eqref{appr2} to
$\mm_\varepsilon$ for the case $\ell = 1$ to get
$$
\sqrt{\mm_\varepsilon(t)} \lesssim \sqrt{\mm_\varepsilon(0)} + \int_0^t
\big(\|u(t)\|_{H^{\alpha^*}} + \|u(t)\|_{H^{\alpha^*}}^3\big)\,dt'.
$$
A limiting argument implies \eqref{appr3}. This completes the proof of Proposition
\ref{mom prop}.
\end{proof}

\section{Appendix}
In this section we provide a proof of the local well-posedness of
Hartree equation \eqref{main eqn}. Here we only assume that $\alpha
> 0$ and $\psi \in L^\infty$.

\begin{proposition}\label{local}
Let $\psi \in L^\infty$. Let $\alpha > 0$, $0 < \gamma < n$ and $n
\ge 1$. Suppose $\varphi \in H^s(\mathbb{R}^n)$ with $s \ge \frac
\gamma 2$. Then there exists a positive time $T$ such that Hartree
equation \eqref{main eqn} has a unique solution $u \in C([0,T]; H^s)
\cap C^1([0, T]; H^{s-\alpha})$. Moreover, if $T^*$ is the maximal existence time and is finite, then $\lim\limits_{t \nearrow
T^*}\|u(t)\|_{H^\frac\gamma2} = \infty$.
\end{proposition}
\begin{proof} We use the standard
contraction mapping argument. So we shall be brief.

Let $(X(T,\rho), d)$ be a complete metric space with metric $d$
defined by
$$
X(T, \rho) = \{u \in L_T^\infty (H^s(\mathbb{R}^n)):\;
\|u\|_{L_T^\infty H^s} \le \rho\}, \;\; d_X(u, v) = \|u -
v\|_{L_T^\infty L^2}.
$$
We define a mapping $\mathcal N: u \mapsto \mathcal N(u)$ on $X(T,
\rho)$ by
\begin{align}\label{nonlinear func}
\mathcal N(u)(t) = U(t)\varphi - i\int_0^t U(t-t')F(u)(t')\,dt',
\end{align}
where $U(t) = e^{-it|\nabla|^\alpha}$.  For $u \in X(T, \rho)$ and
$s \ge \frac\gamma2$ we estimate
\begin{align}\begin{aligned}\label{cont1}
\|\mathcal N(u)\|_{L_T^\infty H^s}&\le \|\varphi\|_{H^s} + T\|F(u)\|_{L_T^\infty H^s}\\
&\lesssim \|\varphi\|_{H^s} +
T\big(\|V_{\gamma}(|u|^2)\|_{L_T^\infty
L^\infty}\|u\|_{L_T^\infty H^s}\\
&\qquad\qquad\qquad\qquad +
\|V_{\gamma}(|u|^2)\|_{L_T^\infty H_{\frac{2n}{\gamma}}^s}\|u\|_{L_T^\infty L^\frac{2n}{n - \gamma}}\big)\\
&\lesssim \|\varphi\|_{H^s} +
T\big(\|V_{\gamma}(|u|^2)\|_{L_T^\infty
L^\infty}\|u\|_{L_T^\infty H^s}\\
&\qquad\qquad\qquad\qquad +
\|V_{\gamma}(\big\langle \nabla\big\rangle^{s}(|u|^2))\|_{L_T^\infty L^{\frac{2n}{\gamma}}}\|u\|_{L_T^\infty L^\frac{2n}{n - \gamma}}\big)\\
&\lesssim \|\varphi\|_{H^s} +
T\big(\|V_{\gamma}(|u|^2)\|_{L_T^\infty
L^\infty}\|u\|_{L_T^\infty H^s}\\
&\qquad\qquad\qquad\qquad +
\|\big\langle \nabla \big\rangle^{s}(|u|^2)\|_{L_T^\infty L^{\frac{2n-\gamma}{2n}}}\|u\|_{L_T^\infty L^\frac{2n}{n - \gamma}}\big)\\
&\lesssim \|\varphi\|_{H^s} + T\big( \|u\|_{L_T^\infty H^\frac
\gamma2}^2\|u\|_{L_T^\infty H^s} + \|u\|_{L_T^\infty
L^\frac{2n}{n-\gamma}}^2\|u\|_{L_T^\infty H^s}\big)\\
&\lesssim \|\varphi\|_{H^s} + T \|u\|_{L_T^\infty
H^\frac\gamma2}^2\|u\|_{L_T^\infty H^s} \lesssim \|\varphi\|_{H^s} +
T\rho^3.
\end{aligned}\end{align}
Here we used the generalized Leibniz rule (Lemma A1-A4 in Appendix of \cite{chwe}) for the second and fifth inequalities, the fractional integration for the fourth one, and the trivial inequality $$V_\gamma = \int_{\mathbb{R}^n}
\frac{\psi(x-y)}{|x-y|^\gamma}|u(y)|^2\,dy \le
\|\psi\|_{L^\infty}\int_{\mathbb R^n} |x-y|^{-\gamma}|u(y)|^2\,dy,$$
%It is well-known that $I_\alpha$ satisfies the inequality (see \cite{st} for instance)
%$$
%\|I_\alpha(\psi)\|_{L^q} \lesssim \|\psi\|_{L^p},\;\frac1q = \frac1p
%- \frac \alpha n, 1 < p < q < \infty.
%$$
 the Hardy-Sobolev inequality
\[
\sup_{x \in \mathbb{R}^n}\left|\int_{\mathbb{R}^n}
\frac{|u(x-y)|^2}{|y|^\gamma}\,dy\right| \lesssim
\|u\|_{\dot{H}^\frac \gamma2}^2,
\]
and the Sobolev embedding $H^\frac \gamma2 \hookrightarrow
L^\frac{2n}{n - \gamma}$ for the last one.
If we choose $\rho$ and $T$ such as $\|\varphi\|_{H^s} \le \rho/2$
and $CT\rho^3 \le \rho/2$, then $\mathcal N$ maps $X(T, \rho)$ to
itself.

Now we show that $\mathcal N$ is a Lipschitz map with a sufficiently
small $T$. Let $u, v \in X(T, \rho)$. Then we have
\begin{align*}
d_X(\mathcal N(u), \mathcal N(v)) &\lesssim
T\left\|V_{\gamma}(|u|^2)u -
V_\gamma(|v|^2)v \right\|_{L_T^\infty L^2}\\
&\lesssim T\big(\left\|V_\gamma(|u|^2)(u - v)\right\|_{L_T^\infty
L^2} + \left\|V_\gamma(|u|^2 -
|v|^2)v\right\|_{L_T^\infty L^2}\big)\\
&\lesssim T\big(\|u\|_{L_T^\infty H^\frac \gamma2}^2d_X(u, v) +
\|V_\gamma(|u|^2 - |v|^2)\|_{L_T^\infty
L^\frac{2n}{\gamma}}\|v\|_{L_T^\infty
L^\frac{2n}{n-\gamma}}\big)\\
&\lesssim T(\rho^2d_X(u, v) + \rho\||u|^2 -
|v|^2\|_{L_T^\infty L^\frac{2n}{2n-\gamma}})\\
&\lesssim T(\rho^2  + \rho(\|u\|_{L_T^\infty L^\frac{2n}{n-\gamma}}
+ \|v\|_{L_T^\infty
L^\frac{2n}{n-\gamma}}))d_X(u, v)\\
&\lesssim T\rho^2 d_X(u, v).
\end{align*}
The above estimate implies that the mapping $\mathcal N$ is a
contraction, if $T$ is sufficiently small.
The uniqueness and time regularity follow easily from the equation
\eqref{main eqn} and a similar contraction argument.

Finally,  let $T^*$ be the maximal existence time. If $T^* <
\infty$, then it is obvious from the estimate \eqref{cont1} and the
standard local well-posedness theory that $ \lim_{t\nearrow
T^*}\|u(t)\|_{H^\frac\gamma2} = \infty$. This completes the proof of
Proposition \ref{local}.
\end{proof}

\section*{Acknowledgments}
The authors would like to thank anonymous referee for his/her
valuable comments. Y. Cho supported by NRF grant 2012-0002855
(Republic of Korea), G. Hwang supported by NRF grant
2012R1A1A1015116, 2012R1A1B3001167 (Republic of Korea),  S. Kwon
partially supported by TJ Park science fellowship and NRF grant
2010-0024017 (Republic of Korea), S. Lee supported  in part by NRF
grant 2009-0083521 (Republic of Korea).
\label{lastpage}
\end{document}